\documentclass[12pt]{amsart}

\usepackage{amscd}
\usepackage{amssymb}
\usepackage{amstext}
\usepackage{amsthm}
\usepackage{mathrsfs}
\usepackage{latexsym}
\usepackage{xcolor}
\numberwithin{equation}{section}
\usepackage{hyperref}

\newcommand{\compl}{\mathbb{C}}

\newcommand{\R}{\mathbb{R}}

\newcommand{\Bn}{\mathbb B_n}
\newcommand{\Sn}{\mathbb S_n}
\newcommand\set[1]{\left\{#1\right\}}
\providecommand{\abs}[1]{\lvert#1\rvert}
\providecommand{\Abs}[1]{\Bigl\lvert#1\Bigr\rvert}

\def\A2w{A^2_\rho}
\def\F2w{\mathcal F^2_\psi}

\providecommand{\normp}[1]{\lVert#1\rVert_\psi}

\def\d{\partial}

\def\idd{i\partial\bar\partial}
\DeclareMathOperator{\Real}{Re}

\newtheorem{Thm}{Theorem}[section]
\newtheorem{theorem}[Thm]{Theorem}
\newtheorem{lemma}[Thm]{Lemma}
\newtheorem{proposition}[Thm]{Proposition}

\newtheorem*{thmA}{Theorem A}

\theoremstyle{definition}
\newtheorem{remark}[Thm]{Remark}
\newtheorem{example}[Thm]{Example}

\begin{document}
\title[On the dimension of the Fock type spaces]{On the dimension of the Fock type spaces}
\author{Alexander Borichev \and Van An Le \and El Hassan Youssfi}
\thanks{The results of Section 2 were obtained in the framework of the project 20-61-46016 by the Russian Science Foundation.\\
A.~Borichev and H.~Youssfi were partially supported by the project ANR-18-CE40-0035.} 

\address{Alexander Borichev: Aix--Marseille University, CNRS, Centrale Marseille, I2M, Marseille, France,\newline
St. Petersburg University, Saint Petersburg, Russia}
\email{alexander.borichev@math.cnrs.fr}
\address{Van An Le: Aix--Marseille University, CNRS, Centrale Marseille, I2M, Marseille, France,\newline
University of Quynhon, Department of Mathematics and Statistics,
170 An Duong Vuong, Quy Nhon, Vietnam}
\email{levanan@qnu.edu.vn}
\address{El Hassan Youssfi: Aix--Marseille University, CNRS, Centrale \newline\noindent Marseille, I2M, Marseille, France}
\email{el-hassan.youssfi@univ-amu.fr}


\begin{abstract} We study the weighted Fock spaces in one and several complex variables. We evaluate the dimension of these 
spaces in terms of the weight function extending and completing earlier results by Rozenblum--Shirokov and Shigekawa.
\end{abstract}

\maketitle

\section{Introduction}

Let $\psi$ be a plurisubharmonic function on $\compl^n$, $n\ge 1$.  The weighted Fock space $\F2w$ is the space of entire functions $f$ such that 
$$ 
\normp{f}^2 =\int_{\compl^n}\abs{f(z)}^2 e^{-\psi(z)}\,dv(z)<\infty, 
$$
where $dv$ is the  volume measure on $\compl^n$.
Note that $\F2w$ is a closed subspace of $L^2(\compl^n, e^{-\psi} \,dv)$ and hence is a Hilbert space  endowed with the inner product 
$$
\langle f,g\rangle_\psi=\int_{\compl^n}f(z)\overline{g(z)}e^{-\psi(z)}\,dv(z),\qquad f,g \in \F2w.
$$

In this paper we study when the space $\F2w$ is of finite dimension depending on the weight $\psi$. 
This problem (at least for the case $n=1$) is motivated by some quantum mechanics questions, especially by the study of zero modes, eigenfunctions with zero eigenvalues.

In \cite[Theorem 3.2]{RoShi}, Rozenblum and Shirokov proposed a sufficient condition for the space $\F2w$ to be of infinite dimension, when $\psi$ is a subharmonic function. 

More precisely, they claimed that if $\psi$ is a finite subharmonic function on the complex plane such that the measure $\mu=\Delta\psi$ is of infinite mass: 
\begin{equation}\label{1}
\mu(\compl)=\int_{\compl}d\mu(z)=\infty,
\end{equation}
then the space $\F2w$ has infinite dimension.

(For the fact that if $\mu=\Delta\psi$ a non-trivial doubling measure, then $\F2w$ has infinite dimension see \cite[Theorem 11.45]{Has}).

We improve and extend somewhat the statement of Rozenblum--Shirokov in our paper, give a necessary and sufficient condition on $\psi$ for the space $\F2w$ to be of finite dimension,  
and calculate this dimension.

The situation is much more complicated in $\compl^n, n\ge 2$. Shigekawa established in \cite{Shigekawa} (see also \cite[Theorem 11.20]{Has} in a book by Haslinger), the following interesting result.

\begin{thmA}  \label{Shi}
Let $\psi: \compl^n \to \R$ be a $\mathcal C^\infty$ smooth function and let $\lambda_0(z)$ be the smallest eigenvalue of the Levi matrix
$$
L_\psi(z)=\idd \psi(z)=\left(\dfrac{\d^2\psi(z)}{\d z_j \d \overline{z_k}}\right)_{j,k=1}^n.
$$
Suppose that 
\begin{equation}\label{2}
\lim _{|z|\to \infty} |z|^2\lambda_0(z)=\infty.
\end{equation}
Then $\dim \F2w=\infty$.
\end{thmA}

Note that the condition \eqref{2} is not necessary. A corresponding example is given in \cite[Section 11.5]{Has} ($\psi(z, w) = |z|^2|w|^2 + |w|^4$). In this paper, we improve Theorem~A by presenting a weaker condition for the dimension of the Fock space $\F2w$ to be infinite. Furthermore, we give several examples that show how far is our condition from being necessary. Finally, we consider several examples (classes of examples) of weight 
functions $\psi$ of special form and evaluate the dimension of $\F2w$.

The rest of the paper is organised as follows. The case of dimension one is considered in Section 2, and the case of higher dimension is considered in Section 3.

\subsubsection*{Acknowledgments}
We thank Friedrich Haslinger and Grigori Rozenblum for helpful remarks.

\section{The case of \texorpdfstring{$\mathbb C$}{C}}

Given a subharmonic function $\psi:\compl\to[-\infty,\infty)$ denote by $\mu_\psi$ the corresponding Riesz measure, $\mu_\psi=\Delta\psi$. 
Next, consider the class $\mathcal M^d$ of the positive $\sigma$-finite atomic measures with masses which are integer multiples of $4\pi$. 
Given a $\sigma$-finite measure $\mu$, consider the corresponding atomic measure $\mu^d$,
$$
\mu^d=\max\Bigl\{\mu_1\in\mathcal M^d:\mu_1\le\mu\Bigr\}.
$$
In fact, for every atom $a\delta_x$ of $\mu$, $\mu^d$ has at the point $x$ an atom of size $4\pi$ times the integer part of $a/(4\pi)$. 
Denote $\mu^c=\mu-\mu^d$, $\mu^d=\sum_k 4\pi\delta_{x_{k,\mu}}$.

Denote by $\mathcal M^c$ the class of the positive $\sigma$-finite measures $\mu$ such that $\mu^d=0$. Note that if 
$\psi$ is finite on the complex plane, then $\mu_\psi$ has no point masses and $\mu_\psi\in \mathcal M^c$. 
Furthermore, if $\mu_\psi\in \mathcal M^c$, then $e^{-\psi}\in L^1_{loc}(v)$.

\begin{lemma}
Let $\psi,\psi_1$ be two subharmonic functions such that $(\mu_\psi)^c=(\mu_{\psi_1})^c$.
Then $\dim\F2w=\dim \mathcal F^2_{\psi_1}$.
\label{lem7}
\end{lemma}

\begin{proof}
Let $F,F_1$ be two entire functions with the zero sets, correspondingly, $\{x_{k,\mu_\psi}\}$ and $\{x_{k,\mu_{\psi_1}}\}$ 
(taking into account the multiplicities). Then 
$\Delta\log|F|^2=(\mu_\psi)^d$, $\Delta\log|F_1|^2=(\mu_{\psi_1})^d$, and the functions $h=\psi-\log|F^2|-\psi^c$, $h_1=\psi_1-\log|F_1^2|-\psi_1^c$ are harmonic. Let $h=\Re H$, $h_1=\Re H_1$ 
for some entire functions $H,H_1$. 

Given an entire function $f$ we have 
\begin{multline*}
f\in \F2w \iff \int_{\compl}|f(z)|^2 e^{-\psi(z)}\,dv(z)<\infty \iff \\
\int_{\compl}|f(z)|^2 e^{-\psi^c(z)-h(z)-\log|F(z)|^2}\,dv(z)<\infty \iff \\
\int_{\compl}|f(z)e^{-H(z)/2}/F(z)|^2 e^{-\psi^c(z)}\,dv(z)<\infty \iff \\
\int_{\compl}|f(z)e^{-H(z)/2}/F(z)|^2 e^{-\psi_1^c(z)}\,dv(z)<\infty \iff \\
\int_{\compl}|f(z)e^{-H(z)/2}/F(z)|^2 e^{-\psi_1(z)+h_1(z)+\log|F_1(z)|^2}\,dv(z)<\infty \iff \\
\int_{\compl}|f(z)e^{-H(z)/2+H_1(z)/2}|F_1(z)/F(z)|^2 e^{-\psi_1(z)}\,dv(z)<\infty \iff \\
f\cdot\frac{F_1}{F}e^{-H/2+H_1/2}\in \mathcal F^2_{\psi_1}.
\end{multline*}
Thus,  $\dim\F2w=\dim \mathcal F^2_{\psi_1}$.
\end{proof}

\begin{lemma}
Let $\psi$ be a subharmonic function such that $\mu_\psi\in \mathcal M^c$. 
If $\dim\F2w<\infty$, then $\mu_\psi(\compl)<\infty$.
\label{rs}
\end{lemma}

See the proof of \cite[Theorem 3.2]{RoShi}. 

\begin{lemma}\label{lem2.1}
Let $\psi$ be a subharmonic function.
Then
$$
\dim\F2w\le \Bigl\lceil \frac{\mu_\psi(\compl)}{4\pi} \Bigr\rceil.
$$
\end{lemma}

Here and later on, given a real number $x$, $\lceil x \rceil$ is the maximal integer smaller than $x$. 

\begin{proof} Set $\mu=\mu_\psi$ and consider 
a modified logarithmic potential $G$ of the measure $\mu$: 
\begin{multline*}
G(z)=\frac{1}{2\pi}\int_{D(0,2)}\log\abs{z-w}\,d\mu(w)+\frac{1}{2\pi}\int_{\compl\setminus D(0,2)}\log\Abs{\dfrac{z-w}{w}}\,d\mu(w)\\
=G_1(z)+G_2(z).
\end{multline*}
Here and later on, $D(z,r)=\{w\in\mathbb C:|w-z|<r\}$. 
Since $\Delta G =\mu=\Delta \psi$, by Lemma~\ref{lem7} we have $\dim\F2w=\dim \mathcal F^2_{G}$. 

Next,
\begin{multline}
\Bigl| G_1(z)-\frac{\mu(D(0,2))}{2\pi} \log|z|\Bigr|\le \frac{1}{2\pi}\int_{D(0,2)} \log \Bigl |1-\frac wz\Bigr| \,d\mu(w)
\\ \le  \frac{C}{|z|},\qquad |z|\ge 4,
\label{star}
\end{multline}
and 
\begin{multline*}
G_2(z)-\frac{\mu(\compl\setminus D(0,2))}{2\pi} \log|z|\\=\frac{1}{2\pi}\int_{\compl\setminus D(0,2)}\log \Bigl| \frac 1z-\frac 1w\Bigr| \,d\mu(w)\le 0,\qquad |z|\ge 4.
\end{multline*}
Thus,
$$
G(z)\le \frac{\mu(\compl)}{2\pi} \log(1+|z|)+\frac{C}{1+|z|},\qquad z\in\compl.
$$

Now, given an entire function $f$, we have 
$$
f\in\F2w\implies \int_{\compl}|f(z)|^2 (1+|z|)^{-\mu(\compl)/(2\pi)}\,dv(z)<\infty.
$$
By a Liouville type theorem, $f$ is a polynomial of degree $N$ such that 
$$
\int_1^\infty r^{2N}r^{-\mu(\compl)/(2\pi)}\,rdr<\infty.
$$
Therefore, $N<-1+\mu(\compl)/(4\pi)$. Thus, $\dim\F2w\le \Bigl\lceil \frac{\mu(\compl)}{4\pi} \Bigr\rceil$.
\end{proof} 

\begin{lemma}\label{lem2.2}
Let $\psi$ be a subharmonic function and suppose that $\mu_\psi\in\mathcal M^c$.
Then
$$
\dim\F2w\ge \Bigl\lceil \frac{\mu_\psi(\compl)}{4\pi} \Bigr\rceil.
$$
\end{lemma}

\begin{proof} 
Set $\mu=\mu_\psi$ and choose $\varepsilon>0$, $R>1$ such that 
$$
\frac{\mu(D(0,R))}{4\pi}>  \Bigl\lceil \frac{\mu (\compl)}{4\pi} \Bigr\rceil +\frac{\varepsilon}{2}.
$$  
Next, increasing $R$, we can guarantee that 
$$
\mu(D(0,R)) >\mu(\compl)-\frac12. 
$$  

Consider 
a modified logarithmic potential $U$ of measure $\mu$: 
\begin{multline*}
U(z)=\frac{1}{2\pi}\int_{D(0,R)}\log\abs{z-w}\,d\mu(w)+\frac{1}{2\pi}\int_{\compl\setminus D(0,R)}\log\Abs{\dfrac{z-w}{w}}\,d\mu(w)\\
=U_1(z)+U_2(z).
\end{multline*}
Since $\Delta U =\mu=\Delta \psi$, by Lemma~\ref{lem7} we have $\dim\F2w=\dim \mathcal F^2_{U}$. 
Arguing as in \eqref{star}, we get 
$$
U_1(z)\ge \frac{\mu(D(0,R))}{2\pi} \log|z| - \frac{C}{|z|},\qquad |z|\ge 2R.
$$
Next, let $|z|\ge 2R$. Then
\begin{multline*}
U_2(z)=\frac{1}{2\pi}\int_{\compl\setminus (D(0,R)\cup D(z,|z|/2))}\log\Abs{\dfrac{z-w}{w}}\,d\mu(w)\\+\frac{1}{2\pi}\int_{D(z,|z|/2)}\log\Abs{\dfrac{z-w}{w}}\,d\mu(w)\\ \ge C-\frac{1}{2\pi}\int_{D(z,|z|/2)}\log\Abs{\dfrac{z/2}{z-w}}\,d\mu(w)=C- U_3(z).
\end{multline*}

Now, we apply a result by Hayman \cite[Lemma~4]{Ha}. The following notation is used there.
Let $\nu$ be a finite positive measure. Given $z\in\mathbb C$, $h>0$, set $n(z,h)=\nu(D(z,h))$, $N(z,h)=\int_{D(z,h)}\log\Bigl|\frac{h}{w-z}\Bigr|\,d\nu(w)$.

\begin{lemma} Let $z_0\in\mathbb C$, $0<d<h/2$. There exists a set $S$ of area at most $\pi d^2$ such that 
$$
N(z,h/2)\le n(z_0,h)\log\frac{16h}{d},\qquad z\in D(z_0,h/2)\setminus S.
$$
\label{hay}
\end{lemma}

Given $m\ge 1$, denote $A_m=\{z\in\compl: 2^m R\le |z|<2^{m+1}R\}$. 
Fix $m\ge 1$ and $k\ge 1$ and apply Lemma~\ref{hay} with $\nu=\mathbf1_{\compl\setminus D(0,R)} \mu$, $2^m R\le |z_0|<2^{m+1}R$, $h=2^{m-1}R$, 
$n(z_0,h)\le 1/2$, and $d=2^{m-k-1}R$ to get for some $C,C_1>0$, $\delta\in(0,1)$:
$$
m_2\bigl\{z\in A_m: U_3(z)> C_1+\delta k\bigr\}\le C\cdot 2^{2m}R^22^{-2k},\quad k\ge 1.
$$
Hence,  
\begin{multline*}
\int_{\compl}(1+|z|)^{-2-\varepsilon}e^{U_3(z)}\,dv(z) \le C+C\sum_{m\ge 1}\sum_{k\ge 1}2^{-(2+\varepsilon)m}e^{\delta k}\\
\times m_2\bigl\{z\in A_m:C_1+\delta k\le U_3(z)< C_1+\delta (k+1)\bigr\}
\\
\le 
C+C\sum_{m\ge 1}\sum_{k\ge 1}2^{-(2+\varepsilon)m}e^{\delta k}2^{2m}R^22^{-2k}<\infty.
\end{multline*}

Next, for every $0\le N\le \Bigl\lceil \frac{\mu(\compl)}{4\pi} \Bigr\rceil-1$ we have 
\begin{multline*}
\int_{\compl}|z|^{2N}e^{-U(z)}\,dv(z)\le 
C\int_{\compl}|z|^{2N}(1+|z|)^{-\mu(D(0,R))/(2\pi)}e^{U_3(z)}\,dv(z)\\ \le 
C\int_{\compl}(1+|z|)^{-2-\varepsilon}e^{U_3(z)}\,dv(z)<\infty
\end{multline*}

Here we use that $\mu_\psi\in\mathcal M^c$ and, hence, $e^{-U}$ is locally integrable.

Finally, we have
$$
\dim\F2w\ge \Bigl\lceil \frac{\mu(\compl)}{4\pi} \Bigr\rceil.
$$
\end{proof} 
 
Summing up Lemmata~\ref{lem7}, \ref{rs}, \ref{lem2.1}, and \ref{lem2.2}, we obtain the following result, extending and slightly correcting \cite[Theorem 3.2]{RoShi}.

\begin{theorem}\label{thm2.2}
Let $\psi$ be a subharmonic function on the complex plane. 
Then the Fock space $\F2w$ is finite-dimensional if and only if 
\begin{equation}
(\mu_\psi)^c(\compl)<\infty.
\label{cond}
\end{equation}
If $\psi$ is finite on $\compl$, then we can write condition \eqref{cond} as  
$\mu_\psi(\compl)<\infty$. 
Finally, if $(\mu_\psi)^c(\compl)<\infty$, then
$$
\dim\F2w= \Bigl\lceil \frac{(\mu_\psi)^c(\compl)}{4\pi} \Bigr\rceil.
$$
\end{theorem}
\bigskip
 
\begin{remark} It is an interesting open question to characterize non subharmonic functions $\psi$ such that the space $\F2w$ is of finite dimension. 
For some results in this direction 
and some physical interpretations see \cite{RoShi1}. 
\end{remark}

\section{The case of \texorpdfstring{$\mathbb C^n$, $n>1$}{Cn, n>1}}

Let $\compl^n$ denote the $n$-dimensional complex Euclidean space. Given $z=(z_1,z_2, \ldots,z_n)\in \compl^n$, we set
$$
\abs{z}=\sqrt{\abs{z_1}^2+\cdots+\abs{z_n}^2}.
$$
Denote $\Bn(z,r)=\{w\in\compl^n:|w-z|<r\}$. Then $\Bn=\Bn(0,1)$ is the unit ball and $\Sn=\partial \Bn$ is the unit sphere in $\compl^n$. Let $d\sigma$ be the normalized surface measure on $\Sn$.

\begin{theorem}\label{dl1}
Let $\psi: \compl^n \to \R$ be a $\mathcal C^2$ smooth function. Given $M>0$, consider $\psi_M(z)=M\log(\abs z^2)$. 
Suppose that for every $M> 0$, the function $\psi-\psi_M$ is plurisubharmonic outside a compact subset of $\compl^n$.
Then $\dim\F2w=\infty$.
\end{theorem}

\begin{proof} We use the fundamental result of Bedford--Taylor \cite{BT} on the solutions of the Dirichlet problem for the complex Monge--Amp\`ere equation. Given $M>0$, choose $r_M>1$ 
such that $\psi-\psi_M$ is plurisubharmonic on $\compl^n\setminus \overline{\Bn(0,r_M)}$. Solving the Dirichlet problem for the complex Monge--Amp\`ere equation on $\Bn(0,r_M)$ with 
the boundary conditions $(\psi-\psi_M)|_{\partial \Bn(0,r_M)}$, we obtain a function $u$. Set 
$$
\widetilde \psi_M(z)=\begin{cases}
(\psi-\psi_M)(z),\qquad z\in \compl^n\setminus \Bn(0,r_M),\\
u(z),\qquad z\in \Bn(0,r_M).
\end{cases}
$$
Then $\widetilde \psi_M$ is a continuous plurisubharmonic function on $\compl^n$ (see also \cite[Section 7]{DE}).

Now, by the H\"ormander theorem (\cite[Theorem 4.4.4]{HO}, see also \cite[Section IV]{BO}), 
there exists an entire function $f\not\equiv0$ such that  
$$
\int_{\compl^n}\abs{f(z)}^2(1+\abs z^2)^{-3n}e^{-\widetilde \psi_M(z)}\,dv(z)<\infty.
$$
Hence, for every $0\le k\le M-\frac32n$, we have 
\begin{multline*}
\int_{\compl^n}\abs{f(z)}^2|z|^{2k}e^{-\psi(z)}\,dv(z)\le 
C+\int_{\compl^n\setminus \Bn(0,r_M)}\abs{f(z)}^2|z|^{2k}e^{-\psi(z)}\,dv(z)\\=
C+\int_{\compl^n\setminus \Bn(0,r_M)}\abs{f(z)}^2|z|^{2k}e^{-\psi_M(z)}e^{-(\psi(z)-\psi_M(z))}\,dv(z)\\ \le 
C+\int_{\compl^n\setminus \Bn(0,r_M)}\abs{f(z)}^2|z|^{-3n}e^{-\widetilde \psi_M(z)}\,dv(z)<\infty.
\end{multline*}
Since $M$ is arbitrary, we have $\dim\F2w=\infty$.
\end{proof}

\begin{remark}
Theorem~A is an immediate corollary of Theorem~\ref{dl1}. \label{r32}

Indeed, an easy computation shows that if $\psi(z)=\varphi(\abs z^2)$, $\varphi \in C^2((0,+\infty))$, then
$$
\dfrac{\d^2\psi}{\d z_j\d\bar{z_k}}(z)=\varphi''(\abs z^2)\bar{z_j}z_k+\varphi'(\abs z^2)\delta_{jk},
$$
where $\delta_{jk}$ is the Kronecker delta symbol. This implies that 
$$
\idd \psi(z)=\varphi'(\abs z^2)I+\varphi''(\abs z^2)z^*z,
$$
where $z^*=\begin{pmatrix}
\bar z_1\\
\dots\\
\bar z_n
\end{pmatrix}$, $z^*z=\bigl[\bar z_jz_k\bigr]_{j,k=1}^n$.
Note also that the spectrum of the matrix $\idd \psi(z)$ is
\begin{equation}
\sigma(\idd \psi(z))=\set{\varphi'(\abs z^2),\varphi'(\abs z^2)+\abs z^2\varphi''(\abs z^2) }.
\label{dop1}
\end{equation}
The first eigenvalue has multiplicity $n-1$ and the second one has multiplicity $1$.

Furthermore, 
\begin{multline*}
L_\psi(z)=\idd \psi(z)=\idd (\psi-\psi_M)(z)+\frac{M}{|z|^2}I-\frac{M}{|z|^4}z^*z\\=
L_{\psi-\psi_M}(z)+\frac{M}{|z|^2}I-\frac{M}{|z|^4}z^*z.
\end{multline*}
Let $z\in \compl^n$ and let 
$V=\begin{pmatrix}
V_1\\
\dots\\
V_n
\end{pmatrix}$
be a normalized eigenvector corresponding to an eigenvalue $\nu$ of $L_{\psi-\psi_M}(z)$. By the hypothesis of Theorem~A, for $|z|>r_M$ we have $\lambda_0(z)|z|^2\ge M$, where 
$\lambda_0(z)$ is the smallest eigenvalue of $L_\psi(z)$. Thus, 
\begin{multline*}
\nu=\langle L_{\psi-\psi_M}(z)V,V\rangle=\langle L_\psi(z)V,V\rangle-\frac{M}{|z|^2}+\frac{M}{|z|^4}\langle z^*z V,V\rangle\\ \ge \lambda_0(z)-\frac{M}{|z|^2}+\frac{M}{|z|^4}|z V|^2\ge 0.
\end{multline*}
Therefore, $\psi-\psi_M$ is plurisubharmonic on $\compl^n\setminus \overline{\Bn(0,r_M)}$, and we are in the conditions of Theorem~\ref{dl1}. \qed
\end{remark}

Now we give an easy example when Theorem~\ref{dl1} applies while Theorem~A does not work.

\begin{example}
Set 
$$
\psi(z)=\varphi(\abs z^2)= \bigl(\log(1+\abs z^2)\bigr)^{3/2},\qquad z\in \compl^n.
$$ 
Then $\varphi(t)=\bigl(\log(1+t)\bigr)^{3/2}$, $t>0$.

Evidently, $\dim \F2w=\infty$. We will show that condition \eqref{2} fails for $\psi$ while the conditions of Theorem~\ref{dl1} are satisfied.

We have 
$$
\varphi'(t)=\frac32\frac1{1+t}\bigl(\log(1+t)\bigr)^{1/2},
$$
and 
$$
\varphi''(t)=-\frac32\frac{\bigl(\log(1+t)\bigr)^{1/2}}{(1+t)^2}+\frac{3}{4(1+t)^2\big(\log(1+t)\big)^{1/2}}.
$$
By \eqref{dop1}, the eigenvalues of the matrix $L_\psi(z)$ are
$$\lambda_1(z)=\dfrac{3\big(\log(1+\abs z^2)\big)^{1/2}}{2(1+\abs z^2)},
$$
and
\begin{multline*}
\lambda_2(z)
	=\dfrac{3\big(\log(1+\abs z^2)\big)^{1/2}}{2(1+\abs z^2)^2}+\dfrac{3\abs z^2}{4(1+\abs z^2)^2\big(\log(1+\abs z^2)\big)^{1/2}}\\
	=\dfrac{3}{4}\dfrac{2\log(1+\abs z^2)+\abs z^2}{(1+\abs z^2)^2\big(\log(1+\abs z^2)\big)^{1/2}}.
\end{multline*}

For $|z|\ge 2$, the smallest eigenvalue of the matrix $L_\psi(z)$ is $\lambda_2(z)$ and
 $$\lim_{\abs z\to \infty}\abs z^2\lambda_2(z)=0.$$
Thus, condition \eqref{2} does not hold.
 
 On the other hand, for $M>0$, the eigenvalues of matrix $L_{\psi-\psi_M}(z)$ are
 $$\alpha_1(z)=\lambda_1(z)-\dfrac{M}{\abs z^2},$$
 and
 $$\alpha_2(z)=\lambda_2(z).$$
 Since $\displaystyle\lim_{\abs z\to \infty}\abs z^2\lambda_1(z)=\infty$ and $\alpha_2(z)>0$, $z\not=0$, the conditions of Theorem~\ref{dl1} are satisfied. \qed
\end{example}

In the rest of the paper we 
show that in different situations the sufficient condition of Theorem~\ref{dl1} is not necessary for $\dim\F2w=\infty$.

\begin{example}
Set 
$$
\psi(z,w)=\abs z^2+2\log(1+\abs w^2), \qquad w, z\in \compl.
$$
It is clear that $\dim \F2w=\infty$. Let us verify that for $M>2$ the function $\psi-\psi_M$ is not plurisubharmonic 
at the points $(1,w)$, $w\in\compl$.

We start with some easy computations:
\begin{gather*}
\frac{\d \psi}{\d z}=\overline{z}, \quad \frac{\d^2 \psi}{\d z \d \overline z}=1, \quad \frac{\d^2 \psi}{\d z \d \overline w}= 0,\\
\frac{\d \psi}{\d w}=\frac{2\overline w}{1+\abs w^2}, \quad \frac{\d^2 \psi}{\d w \d \overline z}= 0, \quad \frac{\d^2 \psi}{\d w \d \overline w}= \frac{2}{(1+\abs w^2)^2}.
\end{gather*}
Now, given $M>0$, we have 
\begin{multline*}
L_{\psi-\psi_M}(z,w)\\=
   \begin{pmatrix}
     1&0\\ 
     0&\frac{2}{(1+\abs w^2)^2}\\ 
   \end{pmatrix}
    +
    \frac{M}{(\abs z^2+\abs w^2)^2}
    \begin{pmatrix}
     \abs z^2&\overline zw\\ 
     z\overline w&\abs w^2\\ 
   \end{pmatrix}
    -
    \frac{M}{\abs z^2+\abs w^2}I \\
  =\begin{pmatrix}
     1-\frac{M\abs w^2}{(\abs z^2+\abs w^2)^2}&\frac{M\overline zw}{(\abs z^2+\abs w^2)^2}\\ 
    \frac{M z\overline w}{(\abs z^2+\abs w^2)^2}&\frac{2}{(1+\abs w^2)^2}-\frac{M\abs z^2}{(\abs z^2+\abs w^2)^2}\\ 
   \end{pmatrix},
\end{multline*}
and, hence,
\begin{multline*}
\det(L_{\psi-\psi_M}(z,w))
	\\=\frac{2}{(1+\abs w^2)^2}-\frac{M\abs z^2}{(\abs z^2+\abs w^2)^2}-\frac{2M\abs w^2}{(1+\abs w^2)^2(\abs z^2+\abs w^2)^2}\\
	=\frac{2(\abs z^2+\abs w^2)^2-M(2\abs w^2+\abs z^2(1+\abs w^2)^2)}{(1+\abs w^2)^2(\abs z^2+\abs w^2)^2}
	<0
\end{multline*}
for $M>2$, $z=1$ and arbitrary $w$. Therefore, the conditions of Theorem~\ref{dl1} do not hold. \qed
\end{example}

\subsection{Weight functions $\psi$ of special form}
In this subsection we evaluate the dimension of $\F2w$ and the applicability of our criterion in Theorem~\ref{dl1}, for some concrete weight functions $\psi$ 
and for $\psi$ in some special classes.


\begin{example}\label{ex4}
Let $k\ge 3$. Set $\psi(z)= \abs{z_1^k+z_2^k}^2$, $z=(z_1,z_2)\in \compl^2$. Given $M>0$, we have  
$$
L_{\psi-\psi_M}(z)=
		\begin{pmatrix}
		k^2\abs{z_1}^{2(k-1)}-\frac{M}{\abs{z}^4}\abs{z_2}^2 & k^2(z_1\overline{z_2})^{k-1}+ \frac{M}{\abs{z}^4} \overline{z_1}z_2\\ 
		k^2(\overline{z_1}{z_2})^{k-1}+ \frac{M}{\abs{z}^4} z_1\overline{z_2} & k^2\abs{z_2}^{2(k-1)}-\frac{M}{\abs{z}^4}\abs{z_1}^2
		\end{pmatrix},
$$
and, hence,
\begin{multline*}
\det(L_{\psi-\psi_M}(z))\\= 
	\left(k^2\abs{z_1}^{2(k-1)}-\frac{M}{\abs{z}^4}\abs{z_2}^2\right)\left(k^2\abs{z_2}^{2(k-1)}-\frac{M}{\abs{z}^4}\abs{z_1}^2\right)\\
	 \qquad - \left(k^2(z_1\overline{z_2})^{k-1}+ \frac{M}{\abs{z}^4} \overline{z_1}z_2\right) \left(k^2(\overline{z_1}{z_2})^{k-1}+ \frac{M}{\abs{z}^4} z_1\overline{z_2}\right)\\
	=-\frac{k^2M}{\abs z^4}\left(\abs{z_1}^{2k}+\abs{z_2}^{2k}+(z_1\overline{z_2})^k+(\overline{z_1}z_2)^k\right)\\
	=-\frac{k^2M}{\abs z^4}\;\abs{z_1^k+z_2^k}^2<0
\end{multline*}
when $z_1^k+z_2^k\not=0$. Thus, for $M>0$, the function $\psi-\psi_M$ is not plurisubharmonic outside a compact subset of $\compl^2$.

Next we are going to verify that $\dim \F2w=\infty$. 

We have 
\begin{multline*}
X:=\int_{\compl^2}e^{-\abs{z_1^k+z_2^k}^2}\,dv(z)
	\asymp\int_0^\infty\int_{\mathbb S_2}r^3e^{-r^{2k}\abs{\zeta_1^k+\zeta_2^k}^2}\,d\sigma(\zeta_1,\zeta_2)\,dr\\
	\asymp \int_{\mathbb S_2}\abs{\zeta_1^k+\zeta_2^k}^{-4/k}\,d\sigma(\zeta_1,\zeta_2).
\end{multline*}
Given $\varepsilon>0$, we consider the set 
$$
T_\varepsilon=\set{(\zeta_1,\zeta_2)\in\mathbb S_2: \abs{\zeta_1^k+\zeta_2^k}<\varepsilon }.
$$
Given $(\zeta_1,\zeta_2)\in \mathbb S_2$ such that $|\zeta_1|\ge |\zeta_2|$, set $\zeta_1=\sqrt{\frac12+r}\cdot e^{i\theta}$ and $\zeta_2=\sqrt{\frac12-r}\cdot e^{i\varphi}$, $r\ge 0$. 
If $(\zeta_1,\zeta_2)\in T_\varepsilon$, then $\abs{\zeta_1}^2-\abs{\zeta_2}^2<C\varepsilon$ for some constant $C=C(k)>0$. Hence, $r\lesssim \varepsilon$. 
Next, since $\abs{\zeta_1^k+\zeta_2^k}<\varepsilon$, we obtain that $|e^{ik\theta}-e^{ik\varphi}|\lesssim \varepsilon$. As a result, we obtain that  
$$
\sigma(T_\varepsilon)\lesssim \varepsilon^2.
$$
Set
$$
U_s=\set{(\zeta_1,\zeta_2)\in\mathbb S_2:2^{-s}< \abs{\zeta_1^k+\zeta_2^k}\le 2^{-s+1}}.
$$

Then 
\begin{multline*}
X\asymp \sum_{s=0}^\infty \int_{U_s}\abs{\zeta_1^k+\zeta_2^k}^{-4/k}\,d\sigma(\zeta_1,\zeta_2)\\
	\lesssim \sum_{s=0}^\infty 2^{-2s}\,2^{4s/k} = \sum_{s=0}^\infty 2^{-2s(1-(2/k))}<\infty, 
\end{multline*}
since $k\ge 3$. Thus, $1\in \F2w$.

In the same way, for every $\alpha>0$ we get 
$$
\int_{\compl^2}e^{-\alpha\abs{z_1^k+z_2^k}^2}\,dv(z)<\infty.
$$
Consider the entire functions $f(z)=e^{\beta(z_1^k+z_2^k)^2}$, $0<\beta<\frac12$. Since
\begin{multline*}
\int_{\compl^2}\bigl|e^{\beta(z_1^k+z_2^k)^2}\bigr|^2e^{-\abs{z_1^k+z_2^k}^2}\,dv(z)
	=\int_{\compl^2}e^{2\beta\Real((z_1^k+z_2^k)^2)-\abs{z_1^k+z_2^k}^2}\,dv(z)\\
	\le \int_{\compl^2}e^{-(1-2\beta)\abs{z_1^k+z_2^k}^2}\,dv(z)<\infty,
\end{multline*}
we conclude that $\dim \F2w=\infty$.
\qed 
\end{example}

Interestingly, $\F2w=0$ if $k=2$. Indeed, let $\psi((z_1,z_2))=|z_1^2+z_2^2|^2$, $f\in \F2w$, $f(z_1,z_2)=(z_1^2+z_2^2)^sg(z_1,z_2)$ for some 
$s\ge 0$, where $g(z_1,z_2)$ is not a multiple of $z_1^2+z_2^2$. By the mean value property, for every $z_1\in\compl\setminus D(0,10)$ 
we have 
\begin{multline*}
|g(z_1,iz_1)|^2\\ \lesssim (1+|z_1|)^2\int_{D(iz_1,2/(1+|z_1|)\setminus D(iz_1,1/(1+|z_1|))}|g(z_1,z_2)|^2e^{-|z_1^2+z_2^2|^2}\,dv(z_2)\\ \lesssim 
(1+|z_1|)^2\int_{D(iz_1,2/(1+|z_1|)\setminus D(iz_1,1/(1+|z_1|))}|f(z_1,z_2)|^2e^{-|z_1^2+z_2^2|^2}\,dv(z_2).
\end{multline*}
Hence,
$$
\int_{\compl} |g(z_1,iz_1)|^2(1+|z_1|)^{-2}\,dv(z_1)\lesssim \|f\|^2_\psi, 
$$
and by a Liouville type theorem,
$g(z,iz)\equiv 0$. Analogously, $g(z,-iz)\equiv 0$. Set $h(z,w)=g(z-iw,z+iw)$. Then $h$ is an entire function and $h(0,w)=h(w,0)\equiv 0$. 
Hence, $h(z,w)=zwh_1(z,w)$ for another entire function $h_1$ and $g(z_1,z_2)=(z_1^2+z_2^2)g_1(z_1,z_2)$ for some entire function $g_1$. This contradiction shows that $\F2w=0$.

Extending the previous example to $\compl^n$ with $n\ge 3$ requires a bit more work.

\begin{example} \label{ex5} Let $n\ge 3$, $k\ge n+1$. 
Set 
$$
\psi(z)= \abs{z_1^k+ \cdots+ z_n^k}^2,\qquad z=(z_1,\ldots,z_n)\in \compl^n.
$$

Let us verify that for $M>0$, the function $\psi-\psi_M$ is not plurisubharmonic outside a compact subset of $\compl^n$.

We have 
\begin{multline*}
L_\psi(z)=k^2\begin{pmatrix}
\abs{z_1}^{2(k-1)} & (z_1\overline{z_2})^{k-1} & \ldots & (z_1\overline{z_n})^{k-1} \\ 
(\overline{z_1}z_2)^{k-1} & \abs{z_2}^{2(k-1)} & \ldots & (z_2\overline{z_n})^{k-1} \\ 
\vdots & \vdots & \ldots & \vdots \\
(\overline{z_1}z_n)^{k-1} & (\overline{z_2}z_n)^{k-1} &\ldots & \abs{z_n}^{2(k-1)}
\end{pmatrix}\\
	=k^2
		\begin{pmatrix}
		z_1^{k-1} \\ 
		z_2^{k-1} \\ 
		\vdots\\
		z_n^{k-1}
		\end{pmatrix} 
		\begin{pmatrix}
		\overline{z_1}^{k-1} & \overline{z_2}^{k-1} & \ldots &\overline{z_n}^{k-1}
		\end{pmatrix}.
\end{multline*}
Set
$$
A(z)=\frac{M}{\abs z^{4}}
	\begin{pmatrix}
		\overline{z_1} \\ 
		\overline{z_2} \\ 
		\vdots\\
		\overline{z_n}
	\end{pmatrix} 
	\begin{pmatrix}
	 	z_1 & z_2 &\ldots & z_n
	\end{pmatrix}. 
$$
Then 
$$
L_{\psi-\psi_M}(z)=L_\psi(z)+A(z)-\frac{M}{\abs z^2}I.
$$
The spectra of the matrices $L_\psi(z)$ and $A(z)$ are
\begin{gather*}
\sigma_{L_\psi(z)}=\set{k^2\bigl(\abs{z_1}^{2(k-1)}+\abs{z_2}^{2(k-1)}+\cdots +\abs{z_n}^{2(k-1)}\bigr),0},\\
\sigma_{A(z)}=\set{\frac{M}{\abs z^2},0}.
\end{gather*}
Let $V$ be the a unit vector in $\compl^n$ orthogonal to 
$\begin{pmatrix}
		z_1^{k-1} \\ 
		z_2^{k-1} \\ 
		\vdots\\
		z_n^{k-1}
		\end{pmatrix}$ and to $\begin{pmatrix}
		\overline{z_1} \\ 
		\overline{z_2} \\ 
		\vdots\\
		\overline{z_n}
	\end{pmatrix} $. 
Then
$$
\langle L_{\psi-\psi_M}(z) V,V\rangle=\langle L_\psi(z)V+A(z)V-\frac{M}{\abs z^2}V,V\rangle=-\frac{M}{\abs z^2}<0.
$$		
Thus, for $M>0$, the function $\psi-\psi_M$ is plurisubharmonic at no points of $\compl^n\setminus\{0\}$.

Finally, let us verify that $\dim \F2w=\infty$. Set
\begin{multline*}
X:=\int_{\compl^n}e^{-\abs{z_1^k+\ldots+z_n^k}^2}\,dv(z)\\
	\asymp \int_0^\infty\int_{\mathbb S_n}r^{2n-1}e^{-r^{2k}\abs{\zeta_1^k+\cdots+\zeta_n^k}^2}\,d\sigma(\zeta_1,\ldots,\zeta_n)\,dr\\
	\asymp \int_{\mathbb S_n}\abs{\zeta_1^k+\ldots+\zeta_n^k}^{-2n/k}\,d\sigma(\zeta_1,\ldots,\zeta_n).
\end{multline*}
Given $\varepsilon>0$, we consider the set 
$$
T_\varepsilon=\set{(\zeta_1,\ldots,\zeta_n)\in\mathbb S_n: \abs{\zeta_1^k+\ldots+\zeta_n^k}<\varepsilon }.
$$
Set 
$$
P(z)= \sum_{j=1}^n z_j^k, \qquad z=(z_1,\ldots,z_n)\in \compl^n.
$$
Then the function $f=\log|P|$ is plurisubharmonic. Following \cite{Kis}, we consider the Lelong number of $f$ at $a\in \compl^n$,
$$
\nu_f(a)=\lim _{r\to 0}\dfrac{\sup _{\abs z\le r}f(a+z)}{\log r}\in [0,\infty].
$$
If $f(a)\ne 0$, then $\nu_f(a)=0$. Otherwise, let $a= (a_1, \ldots, a_n)\ne 0$ and $f(a)=0$. Without loss of generality, we can assume that $a_1\ne 0$. If $0<r<\frac{\abs{a_1}}{2}$, then 
$$
f\bigl(a+(r,0,\ldots,0)\bigr)=\log\abs{(a_1+r)^k-a_1^k}= \log\abs{ka_1^{k-1}r+O(r^2)}, \quad r\to 0,
$$
and hence, $\nu_f(a)=1$. By Theorem 3.1 in \cite{Kis}, applied to $\Omega=2\Bn$, $K=\overline{\Bn}\setminus \frac12\Bn$, $1<\alpha<2$, we obtain
\begin{multline*}
v\bigl(\{z\in K: \abs{P(z)}\le e^{-u}\}\bigr)
	=v\bigl(\{z\in K: f(z)\le {-u}\}\bigr)\\
	\le C_{\alpha}e^{-\alpha u}, \qquad u\ge 0.
\end{multline*}
By homogeneity of $P$, 
$$
\sigma(T_\varepsilon)\le C\varepsilon^\alpha,\qquad \varepsilon>0,
$$
for some constant $C>0$.

Arguing as in Example~\ref{ex4}, we obtain first that $1\in \F2w$ and then that $\dim \F2w=\infty$ for $k\ge n+1$. \qed
\end{example}

At the end of the paper, we consider two special classes of weight functions $\psi$: radial weight functions and the functions of the form $\psi(z_1,\ldots,z_n)=\sum_{j=1}^n
\psi_j(z_j)$.

Suppose that $\psi(z)=\varphi(|z|^2)$ is a radial plurisubharmonic function of class $C^2$. By the computations in Remark~\ref{r32}, 
\begin{equation}
\dfrac{\d^2\psi}{\d z_j\d\bar{z_k}}(z)=\varphi''(\abs z^2)\bar{z_j}z_k+\varphi'(\abs z^2)\delta_{jk}.
\label{5star}
\end{equation}
The action of the Monge--Amp\`ere operator on $\psi$ is 
\begin{multline*}
(dd^c\psi)^n=4n!\det\Bigl(\dfrac{\d^2\psi}{\d z_j\d\bar{z_k}}\Bigr)\,dv\\=
4n!(\varphi'(|z|^2))^{n-1}(\varphi'(|z|^2)+|z|^2\varphi''(|z|^2))\,dv.
\end{multline*}

\begin{proposition}\label{prop}
Suppose that $\psi(z)=\varphi(|z|^2)$ is a radial plurisubharmonic function of class $C^2$. Then $\dim \F2w=\infty$ if and only if 
\begin{equation}
\int_{\mathbb C^n}(dd^c\psi)^n=\infty.
\label{2star}
\end{equation}
\end{proposition}

\begin{proof}
Since the spectrum of the matrix \eqref{5star} consists of the eigenvalues  
$\varphi'(\abs z^2)$ and $\varphi'(\abs z^2)+\abs z^2\varphi''(\abs z^2)$, 
the first eigenvalue has multiplicity $n-1$ and the second one has multiplicity $1$, we have $\varphi'\ge 0$, 
$(r\varphi'(r))'\ge 0$ on $\mathbb R_+$. Furthermore, we have 
\begin{multline*}
\int_{\mathbb C^n}(dd^c\psi)^n=C\int_0^\infty (\varphi'(r^2))^{n-1}(\varphi'(r^2)+r^2\varphi''(r^2))\,dr^{2n}
\\=C\int_0^\infty d\bigl((r\varphi'(r))^n\bigr).
\end{multline*}
Thus, \eqref{2star} is equivalent to the relation $\lim_{r\to\infty}r\varphi'(r)=\infty$. Now, if $r\varphi'(r)$ is bounded on $\mathbb R_+$, then 
$\psi(z)=O(\log|z|)$, $|z|\to\infty$, and a version of the Liouville theorem shows that $\dim \F2w<\infty$. On the other hand, if 
$\lim_{r\to\infty}r\varphi'(r)=\infty$, then $\log|z|=o(\psi(z))$, $|z|\to\infty$, and the polynomials belong to $\F2w$. Hence, $\dim \F2w=\infty$.
\end{proof}

For general $C^2$ plurisubharmonic functions, the radial case suggests the following question. 
Is it true that $\dim \F2w=\infty$ if and only if \eqref{2star} holds? Our last example gives a negative answer to this question.

\begin{example} Given subharmonic functions $\psi_j$ on the complex plane, $1\le j\le n$, set
\begin{equation}
\psi(z_1,\ldots,z_n)=\sum_{j=1}^n\psi_j(z_j).
\label{1star}
\end{equation}

Claim: $\dim \F2w<\infty$ if and only if either $\max_j\dim \mathcal F^2_{\psi_j}<\infty$ or $\min_j\dim \mathcal F^2_{\psi_j}=0$.

In one direction, by the Fubini theorem, if $\dim \F2w<\infty$, then $\max_j\dim \mathcal F^2_{\psi_j}<\infty$ or $\min_j\dim \mathcal F^2_{\psi_j}=0$.
In the opposite direction, it is clear that if $\min_j\dim \mathcal F^2_{\psi_j}=0$, then $\F2w=0$. It remains to verify that 
if $\max_j\dim \mathcal F^2_{\psi_j}<\infty$, then $\dim \F2w<\infty$.

First, suppose that $n=2$, $\dim \mathcal F^2_{\psi_1}<\infty$, $N=\dim \mathcal F^2_{\psi_2}<\infty$. Fix 
a basis $(g_k)$, $1\le k\le N$, in the space $\mathcal F^2_{\psi_2}$ and choose a family of points $(w_m)$, $1\le m\le N$, such that $\det Q\not=0$, where 
$Q=\bigl(g_k(w_m)\bigr)_{k,m=1}^N$. 

Next, choose $f\in \F2w$. By the mean value property, 
$$
|f(z,w)|^2\le \frac1\pi \int_{D(z,1)}|f(\zeta,w|^2\,dv(\zeta),\qquad z,w\in\compl.
$$
Therefore, for every $z\in\compl$, the function $f(z,\cdot)$ belongs to $\mathcal F^2_{\psi_2}$, and, hence, we have  
$$
f(z,\cdot)=\sum_{k=1}^Na_k(z)g_k.
$$
In the same way, the functions $f(\cdot,w_j)$, $1\le j\le N$, belong to $\mathcal F^2_{\psi_1}$.

Next,
$$
Q^{-1}\begin{pmatrix}
		f(z,w_1) \\  
		\vdots\\
		f(z,w_N)
		\end{pmatrix} =\begin{pmatrix}
		a_1(z) \\ 
		\vdots\\
		a_N(z)
	\end{pmatrix}. 
$$
Hence, every $a_j$ belongs to $\mathcal F^2_{\psi_1}$. Since $\dim \mathcal F^2_{\psi_1}<\infty$, 
we conclude that the space $\F2w$ has finite dimension. 
For $n\ge2$ we can just use an inductive argument. This completes the proof of Claim.

Let us return to general $\psi$ satisfying \eqref{1star}. We have 
$$
\int_{\mathbb C^n}(dd^c\psi)^n=C\int_{\mathbb C^n}\prod_{j=1}^n \Delta\psi_j(z_j)\,dv(z)
=C\prod_{j=1}^n\int_{\mathbb C} \Delta\psi_j(z_j)\,dv(z_j).
$$

Now, if $n=2$, $\psi_1(z)=|z|^2$, $\Delta \psi_2(z)=\max(1-|z|,0)$, then 
$$
\int_{\mathbb C^n}(dd^c\psi)^n=\infty,
$$ 
but $\F2w=0$. Thus, Proposition~\ref{prop} does not extend to general $C^2$-smooth plurisubharmonic functions. 

\end{example}



\begin{thebibliography}{1}

\bibitem{BT} E.~Bedford, B.~A.~Taylor, \textit{The Dirichlet problem for a complex Monge--Amp\`ere equation}, Invent.\ Math.\ \textbf{37} (1976), 1--44.

\bibitem{BO} E.~Bombieri, \textit{Algebraic values of meromorphic maps}, Invent.\ Math.\ \textbf{10} (1970), 267--287.

\bibitem{DE} J.-P.~Demailly, \textit{Potential Theory in Several Complex Variables}, Cours donn\'e dans le cadre de l'Ecole d'\'et\'e d'Analyse Complexe organis\'ee par le CIMPA, Nice, Juillet 1989,
Manuscript available at \text{\tt www-fourier.ujf-grenoble.fr/$\sim$demailly/manuscripts/nice\_cimpa.pdf} 


\bibitem{Has} F.~Haslinger, \textit{Complex analysis. A functional analytic approach}, 
De Gruyter Graduate, Berlin, 2018.

\bibitem{Ha} W.~Hayman, \textit{The minimum modulus of large integral functions}, Proc.\ London Math.\ Soc.\ (3) \textbf{2} (1952), 469--512.

\bibitem{HO} L.~H\"ormander, \textit{An introduction to complex analysis in several variables}, 
Third edition. North-Holland Mathematical Library, \textbf{7}. North-Holland Publishing Co., Amsterdam, 1990.  

\bibitem{Kis} Ch.~Kiselman, \textit{Ensembles de sous-niveau et images inverses des fonctions plurisousharmoniques}, Bull.\ Sci.\ Math.\ \textbf{124} (2000), 75--92.

\bibitem{RoShi}  G.~Rozenblum, N.~Shirokov, \textit{Infiniteness of zero modes for the Pauli operator with singular magnetic field},
J. Func.\ Anal.\ \textbf{233} (2006), 135--172.

\bibitem{RoShi1}  G.~Rozenblum, N.~Shirokov, \textit{Entire functions in weighted $L^2$ and zero modes of the Pauli operator with non-sign definite magnetic field},
Cubo \textbf{12} (2010), 115--132.

\bibitem{Shigekawa} I.~Shigekawa, \textit{Spectral properties of Schr\"odinger operators with magnetic fields for a spin $1/2$
particle}, J. Func.\ Anal.\ \textbf{101} (1991), 255--285.


\end{thebibliography}

\end{document}